\documentclass{amsart}
\usepackage{graphicx}
\usepackage{amssymb}
\usepackage[shortlabels]{enumitem}

\newtheorem{thm}{Theorem}[section]

\newtheorem{lem}[thm]{Lemma}

\theoremstyle{definition}

\theoremstyle{remark}
\newtheorem{rem}[thm]{Remark}

\begin{document}

\title[half-space Bernstein theorem]{A half-space Bernstein theorem for anisotropic minimal graphs}
\author{Wenkui Du}
\author{Connor Mooney}
\author{Yang Yang}
\author{Jingze Zhu}
\begin{abstract}
We prove that an anisotropic minimal graph over a half-space with flat boundary must itself be flat. This generalizes a result of Edelen-Wang to the anisotropic case. The proof uses only the maximum principle and ideas from fully nonlinear PDE theory in lieu of a monotonicity formula.
\end{abstract}
\maketitle

\section{Introduction}

In this paper we prove that if $\Sigma$ is an anisotropic minimal graph over a half-space and $\partial \Sigma$ is flat, then $\Sigma$ is flat. More generally, we prove that if $\Sigma$ is an anisotropic minimal graph over a convex domain that is not the whole space and $\Sigma$ has linear boundary data, then $\Sigma$ is flat.

We now state the result more precisely. We assume that $\Sigma \subset \mathbb{R}^{n+1}$ is the graph of a function $u \in C^{\infty}(\Omega) \cap C(\overline{\Omega})$, where $\Omega \subset \mathbb{R}^n$ is a convex domain that is not the whole space. We assume further that $u|_{\partial \Omega}$ agrees with a linear function $L$. Finally, we assume that $\Sigma$ is a critical point of the functional
\begin{equation}\label{PEF}
A_{\Phi}(\Sigma) := \int_{\Sigma}\Phi(\nu)\,d\mathcal{H}^{n}
\end{equation}
where $\mathcal{H}^n$ is $n$-dimensional Hausdorff measure, $\nu$ is the upper unit normal to $\Sigma$, and $\Phi$ is the support function of a smooth, bounded, uniformly convex set $K$ (the Wulff shape). We prove:

\begin{thm}\label{main}
Under the above conditions, $u$ is linear.
\end{thm}

We note that Theorem \ref{main} holds in all dimensions, in contrast with Bernstein-type results for entire anisotropic minimal graphs (linearity is only guaranteed when $n \leq 3$ for general anisotropic functionals (see \cite {Si2}, \cite{MY2}, \cite{M}, \cite{MY}), and only when $n \leq 7$ in the case of the area functional $K = B_1$ (see \cite{S}, \cite{BDG})). The linearity of the boundary data is thus quite powerful for rigidity.

Functionals of the form (\ref{PEF}) are well-studied, both as natural generalizations of the area functional and as models e.g. of crystal formation (\cite{ASS}, \cite{DDG}, \cite{DT}, \cite{FM}, \cite{FMP}). From a technical perspective, what distinguishes general anisotropic functionals from the area case is the absence of a monotonicity formula (\cite{All}), so one cannot reduce regularity and Bernstein-type problems to the classification of cones. This requires the development of more general and sophisticated approaches. In the case of the area functional Theorem \ref{main} was proven by Edelen-Wang in \cite{EW}, and the monotonicity formula played an important role in the proof (particularly in the case that $\Omega$ is a half-space). In contrast, we use only the maximum principle and ideas from fully nonlinear PDE theory, namely, an ABP-type measure estimate (Lemma \ref{ABP}) and an argument reminiscent of the proof of Krylov's boundary Harnack inequality (Lemma \ref{Trapping}), as exposed e.g. in Section 3 of \cite{MNotes}.

\vspace{5mm}

The paper is organized as follows. In Section \ref{Prelim} we recall some useful facts about anisotropic functionals and about the minimizing properties of anisotropic minimal graphs, and we prove an ABP-type measure estimate. In Section \ref{Proof} we prove Theorem \ref{main}.

\section*{Acknowledgements}
W. Du appreciates the support from the NSERC Discovery Grant RGPIN-2019- 06912 of Prof. Y. Liokumovich at the University of Toronto.
C. Mooney was supported by a Sloan Fellowship, a UC Irvine Chancellor's Fellowship, and NSF CAREER Grant DMS-2143668.
Y. Yang gratefully acknowledges the support of the Johns Hopkins University Provost's Postdoctoral Fellowship Program.

\section{Preliminaries}\label{Prelim}

\subsection{Anisotropic Minimal Surfaces}
First we recall a few useful identities related to the integrand $\Phi$. First, we have
\begin{equation}\label{K}
K = \nabla \Phi(\mathbb{S}^{n}), \quad \nu_K(\nabla \Phi(x)) = x
\end{equation}
for $x \in \mathbb{S}^n$. Here $\nu_K$ is the outer unit normal to $K$. The second identity can be seen using the one-homogeneity of $\Phi$, which implies that $x$ is in the kernel of $D^2\Phi(x)$ for all $x \in \mathbb{R}^{n+1} \backslash \{0\}$. Differentiating the second identity we see that
\begin{equation}\label{2FF}
II_K(\nabla \Phi(x)) = (D_T^2\Phi)^{-1}(x), \quad x \in \mathbb{S}^n.
\end{equation}
Here $D_T^2\Phi(x)$ is the Hessian of $\Phi$ on the tangent plane to $\mathbb{S}^n$ at $x$, and here and below, $II_S$ denotes the second fundamental form of a hypersurface $S$. 

Next we recall that if $S$ is a critical point of $A_{\Phi}$ with unit normal $\nu_S$, then the Euler-Lagrange equation reads
\begin{equation}\label{EL}
\text{tr}(D_T^2\Phi(\nu_S(x))II_{S}(x)) = 0.
\end{equation}
The property of being a critical point of $A_{\Phi}$ is dilation and translation invariant. Furthermore, isometries of $S$ and $\nu_S$ by elements of $O(n+1)$ are critical points of anisotropic functionals obtained by performing the same isometries of $K$, and flipping the unit normal of $S$ gives a critical point of the anisotropic functional obtained by replacing $K$ with $-K$.

When $S$ is the graph of a function $w$, and $\nu_S$ is the upper unit normal, the Euler-Lagrange equation (\ref{EL}) can be written
\begin{equation}\label{GraphEL}
\text{tr}(D^2\phi(\nabla w)D^2w) = 0,
\end{equation}
where $\phi(z) = \Phi(-z,\,1)$ for $z \in \mathbb{R}^n$. It follows from (\ref{2FF}) and the uniform convexity of $K$ that the equation (\ref{GraphEL}) is uniformly elliptic when $\nabla w$ is bounded.

\subsection{Minimizing Properties of Graphs}
We will use the following minimizing property of anisotropic minimal graphs. Let $\Omega_S \subset \mathbb{R}^n$ be any domain and let $S$ be a critical point of $A_{\Phi}$ given by the graph of a function $w \in C^{\infty}(\Omega_S) \cap C(\overline{\Omega_S})$, with upper unit normal $\nu_S$. Let $E := \{x_{n+1} \leq w(x),\, x \in \overline{\Omega_S}\}$ be the subgraph of $w$. Finally, let $U \subset \mathbb{R}^{n+1}$ be any bounded open set that doesn't intersect the vertical sides $\{x_{n+1} \leq w(x),\, x \in \partial\Omega_S\}$. Then for any $U' \subset \subset U$, the anisotropic perimeter of $E \backslash U'$ (with respect to the outer unit normal) in $U$ is at least the anisotropic area of $S$ in $U$. This follows quickly from the observation that the vector field $\nabla\Phi(\nu_S)$ in the cylinder over $\Omega_S$, extended to be constant in the $x_{n+1}$ direction, is a calibration. Indeed, it is divergence-free (this follows from the Euler-Lagrange equation (\ref{EL})), and satisfies $\nabla \Phi(\nu_S) \cdot a \leq \Phi(a)$ for all $a \in \mathbb{S}^n$, since $\Phi(a) = \max_{b \in \mathbb{S}^n} \nabla \Phi(b) \cdot a$.

\subsection{Measure Estimate}
Now we prove an ABP-type measure estimate reminiscent of the first step in the proof of the Krylov-Safonov Harnack inequality. The difference is that we do not deal with graphs. The following result is a generalization to the anisotropic case of a lemma proved for minimal surfaces in \cite{Sav}. 

We first set some notation. We let $B_r(x)$ denote a ball in $\mathbb{R}^n$. We define $Q_{r,\,s,\,t}(x) \subset \mathbb{R}^{n+1}$ to be the cylinder $B_r(x) \times (s,\,t)$. For $\lambda \in (0,\,1)$ we let the minimal Pucci operator $\mathcal{M}^-_{\lambda}$ on symmetric $n \times n$ matrices be defined by $\lambda$ times the sum of positive eigenvalues plus $\lambda^{-1}$ times the sum of negative eigenvalues.

The following lemma says that if an anisotropic minimal surface contained on one side of a hyperplane is very close at a point to the hyperplane, then it is very close at most points.

\begin{lem}\label{ABP}
Assume that $S$ is a smooth critical point of $A_{\Phi}$ given by the boundary of a set $E \subset Q_{1,0,1}(0)$. For all $\delta > 0$ small, there exists $\epsilon_0(\delta,\,n,\,K) > 0$ such that if $\epsilon e_{n+1} \in S$ and $\epsilon < \epsilon_0$, then $S$ contains (and lies above) the graph of a function $w$ on a set $G \subset B_{1/3}(0)$ such that 
$$|G| \geq |B_{1/3}| - C_1\delta^{1/2},\, 0 < w < C_1\delta^{3/2}, \text{ and } |\nabla w| < C_1\delta^{1/2}.$$
Here $C_1$ depends only on $n,\,K$.
\end{lem}
\begin{proof}
We may assume that the unit normal to $S$ is the inner unit normal to $E$, after possibly replacing $K$ by $-K$. We claim that in each vertical cylinder over a ball of radius $\delta$ contained in $B_{1/3}(0)$ there is some point in $S$ a distance at most $C_0(\delta,\,n,\,K)\epsilon$ from $\{x_{n+1} = 0\}$. Let $\lambda(K)$ be small enough that the eigenvalues of $D^2\phi$ are between $\lambda$ and $\lambda^{-1}$ in $B_1$, where $\phi(z) = \Phi(-z,\,1)$. We can choose $M(n,\,K)$ large so that for 
$$\varphi_0 := \min\{|x|^{-M},\,\delta^{-M}\}-(3/2)^M$$ 
we have $\mathcal{M}_{\lambda}^-(D^2 \varphi_0) > 0$ outside of $B_{\delta}$ and $\varphi_0 > 1$ on $\partial B_{1/3}$. If $\epsilon_0(\delta,\,n,\,K)$ is small and $\epsilon < \epsilon_0$, then $\epsilon|\nabla\varphi_0| < 1$ outside $B_{\delta}$, hence $\epsilon\varphi_0$ is a sub-solution to (\ref{GraphEL}) outside $B_{\delta}$. If the claim in the second sentence of the proof was false in the cylinder over some ball $B_{\delta}(x_0)$ for $C_0(\delta,\,n,\,K)$ sufficiently large, then we can slide the graph of $\epsilon\varphi_0(\cdot - x_0)$ from below until it touches $S$ from one side outside of the cylinder over $B_{\delta}(x_0)$ (see Figure \ref{Fig1}), and at the contact point we violate the equation (\ref{GraphEL}).

\begin{figure}
 \begin{center}
    \includegraphics[scale=0.7, trim={20mm 170mm 0mm 20mm}, clip]{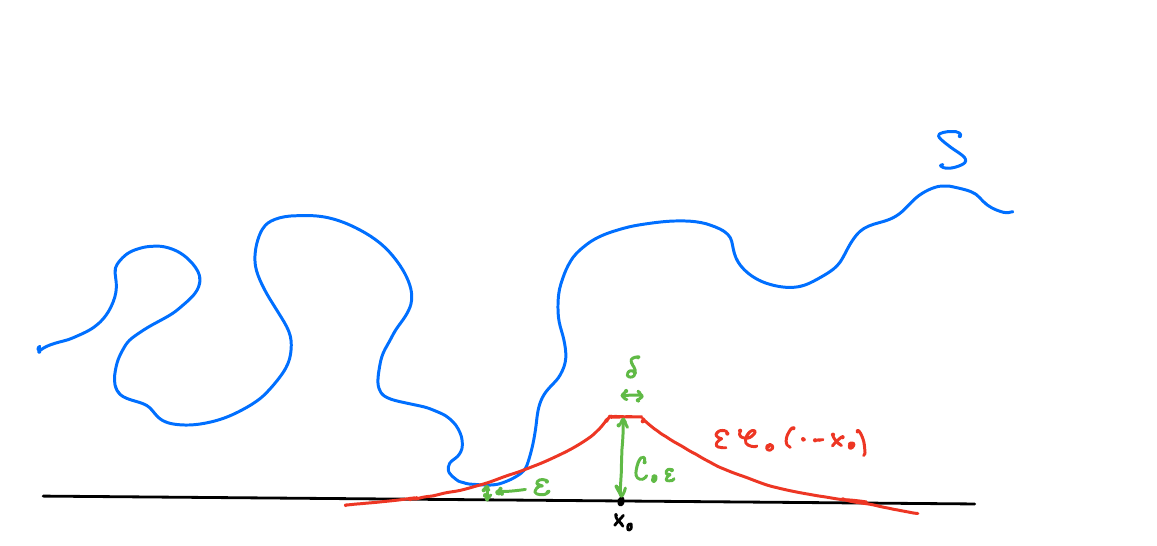}
\caption{$S$ gets $C_0\epsilon$ close to $\{x_{n+1} = 0\}$ at scale $\delta$.}
\label{Fig1}
\end{center}
\end{figure}

Up to taking $\epsilon(\delta)$ smaller we may assume that $C_0\epsilon \leq \delta^{3/2}$. Below $C_i,\, i \geq 1$ will denote large constants depending on $K$. We let $r  = C_2\delta^{1/2}$ and we slide copies of $rK$ centered over points in $B_{1/3-C_3\delta^{1/2}}$ from below until they touch $S$. By the first step, we can take $C_2,\,C_3$ such that the contact happens at points $x$ that are in the cylinder over $B_{1/3}$ and in $\{x_{n+1} < C_1\delta^{3/2}\}$, with upper unit normal $\nu(x)$ lying within $C_1\delta^{1/2}$ of $e_{n+1}$ (see Figure \ref{Fig2}). Here we are using that $K$ is smooth and uniformly convex, hence has interior and exterior tangent spheres of universal radii (depending only on $K$) at all points on its boundary. The corresponding centers $y$ can be found by the relation 
$$y = x - r\nabla \Phi(\nu(x)).$$
Differentiating in $x$ gives
$$D_xy = I + rD_T^2\Phi(\nu(x))II_S(x).$$
Since the second fundamental form of the rescaled Wulff shape at the contact point $x$ is $r^{-1}(D_T^2\Phi)^{-1}(\nu(x))$ (see (\ref{2FF})), we have $II_S(x) \geq -r^{-1}(D_T^2\Phi)^{-1}(\nu(x))$, whence $D_xy \geq 0.$ Since the second term is trace-free we have by the AGM inequality that
$$\det D_xy \leq 1.$$
Thus, the infinitesimal surface measure of centers $y$ is smaller than that of contact points $x$. Since the tangent plane to the surface of contact points at $x$ and the surface of centers at $y$ is the same, the same inequality holds under projection in the $x_{n+1}$ direction. Applying the area formula and recalling that the centers project in the $x_{n+1}$ direction to $B_{1/3-C_3\delta^{1/2}}$ completes the proof.

\begin{figure}
 \begin{center}
    \includegraphics[scale=0.7, trim={20mm 160mm 0mm 30mm}, clip]{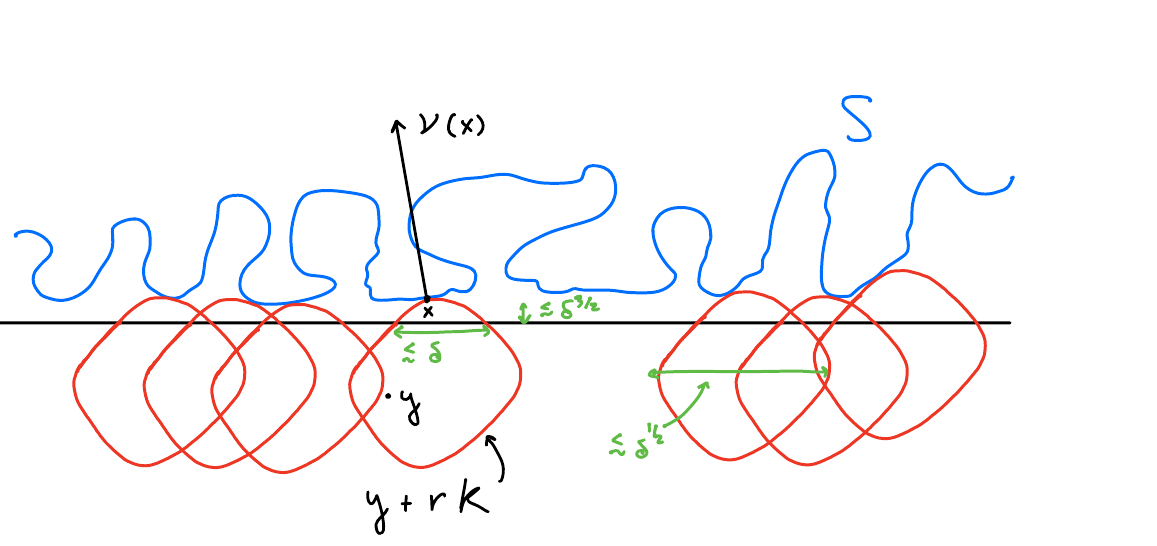}
\caption{The contact points between copies of $rK$ slid from below and $S$ project to nearly the whole ball.}
\label{Fig2}
\end{center}
\end{figure}

\end{proof}

\section{Proof}\label{Proof}
Before proving Theorem \ref{main} we establish some notation. After performing rigid motions, we may assume that $\Omega \subset \{x_1 > 0\}$, that $\{x_1 = 0\}$ is tangent to $\partial \Omega$ at the origin, and that $L(0) = 0$. We let 
$$\Gamma = \text{graph}(L) \cap \{x_1 = 0\}.$$ 
There are three possibilities to consider:
\begin{enumerate}[(A)]
\item $\Omega = \{x_1 > 0\}$ (half-space case)
\item $\Omega = \{0 < x_1 < c < \infty\}$ (slab case)
\item $\overline{\Omega} \cap \{x_1 = 0\} \neq \{x_1 = 0\}$.
\end{enumerate}

We define
\begin{equation}\label{Slopes}
A_+ := \inf\{A : u \leq L + Ax_1 \text{ in } \overline{\Omega}\}, \quad A_- := \sup\{A : u \geq L + Ax_1 \text{ in } \overline{\Omega}\}
\end{equation}
where $A_+ \in \mathbb{R} \cup \{+\infty\}$ and $A_- \in \mathbb{R} \cup \{-\infty\}$. It is clear that $A_- \leq A_+$, and that $A_- \leq 0 \leq A_+$ in cases (B) and (C). To prove Theorem \ref{main} it suffices to prove that $A_+ = A_-$.

We let $H_{\pm}$ denote the graphs of $L + A_{\pm}x_1$ in $\{x_1 \geq 0\}$. When $A_+ = \infty$ we interpret $H_+$ as the closed half-space in $\{x_1 = 0\}$ lying above $\Gamma$, and we understand $H_-$ similarly when $A_- = -\infty$. Finally, we let
$$\Sigma_k := k^{-1}\Sigma \text{ and } u_k := k^{-1}u(k \cdot),$$
so that $\Sigma_k$ are the graphs of $u_k$.

The following is a version of the Hopf lemma, and is reminiscent of a step in the proof of Krylov's boundary Harnack inequality.

\begin{lem}\label{Trapping}
Assume that $A_-$ is anything in case (A) and nonzero in case (B) or (C). Then
$\Sigma_k$ contain points that converge as $k \rightarrow \infty$ to a point in $H_- \backslash \Gamma$. The same statement holds with ``$-$" replaced by ``$+$".
\end{lem}
\begin{proof}
Assume that $\Sigma_k$ do not contain points that converge to something in $H_- \backslash \Gamma$. Then some subsequence $\{\Sigma_{k_j}\}$ 
avoids a neighborhood of the point in $H_- \backslash \Gamma$ that is unit distance from the origin and orthogonal to $\Gamma$. We can find barriers similar to the one in the proof of Lemma \ref{ABP} that are graphs over $H_-$, bound all $\Sigma_{k_j}$ from below, and meet $\Gamma$ at a positive angle (see Figure \ref{Fig3}) to conclude that
$$u_{k_j} \geq L + \begin{cases}
(A_- + \epsilon)x_1, \quad A_- \in \mathbb{R} \\
-\epsilon^{-1}x_1, \quad A_- = -\infty
\end{cases}$$
in $B_{\delta} \cap k_j^{-1}\Omega$ for some $\epsilon,\,\delta > 0$. Here we used that $A_- < 0$ in cases (B) and (C) to guarantee that the barriers lie below $\Sigma_{k_j}$ on the boundaries of $\Sigma_{k_j}$. From the definition of $u_{k_j}$ and the invariance of the right hand side of the above inequality under Lipschitz rescalings, we see that the same inequality holds for $u$ in $B_{k_j\delta} \cap \Omega$. After taking $j \rightarrow \infty$, we contradict the definition of $A_-$. After reflecting over $\{x_{n+1} = 0\}$, the same argument shows the result with ``$-$" replaced by ``$+$".

\begin{figure}
 \begin{center}
    \includegraphics[scale=0.8, trim={15mm 160mm 0mm 30mm}, clip]{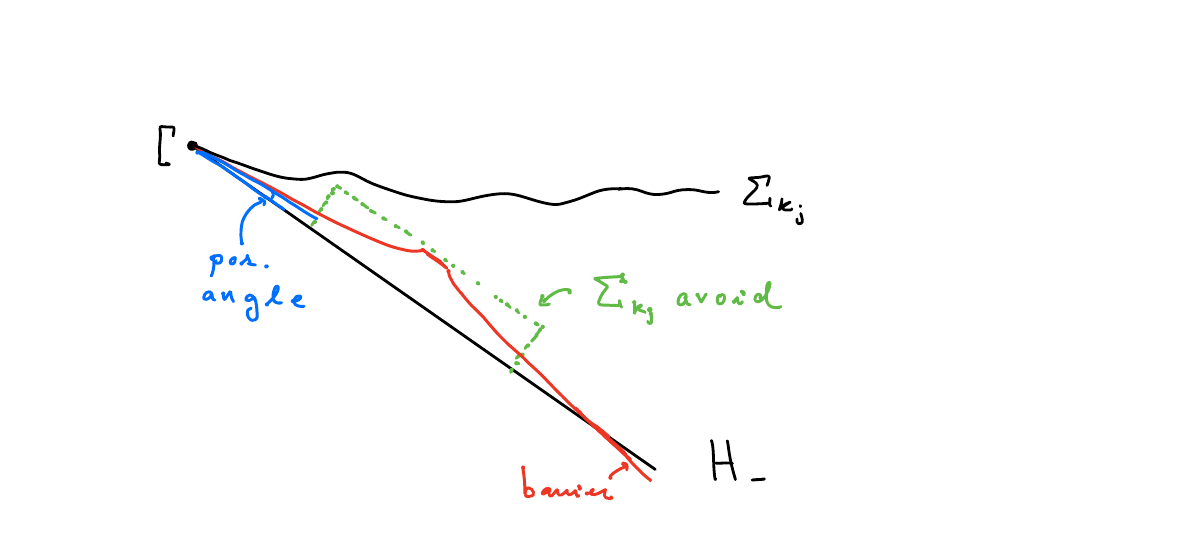}
\caption{Hopf lemma type barriers.}
\label{Fig3}
\end{center}
\end{figure}

\end{proof}

\begin{proof}[{\bf Proof of Theorem \ref{main}}]
We first treat case (C). If $A_+ > 0$, then by Lemma \ref{Trapping} and Lemma \ref{ABP} appropriately rescaled (in fact, just the proof of the first part using barriers) we get that $\Sigma_k$ contains points close to $H_+$ that don't project in the $x_{n+1}$ direction to $\overline{\Omega} \supset k^{-1}\overline{\Omega}$ for some $k$ large, a contradiction of graphicality (see Figure \ref{Fig4}). We conclude that $A_+ = 0$. The assertion that $A_- = 0$ follows from the same argument, after reflection over $\{x_{n+1} = 0\}$.

\begin{figure}
 \begin{center}
    \includegraphics[scale=0.8, trim={30mm 170mm 0mm 20mm}, clip]{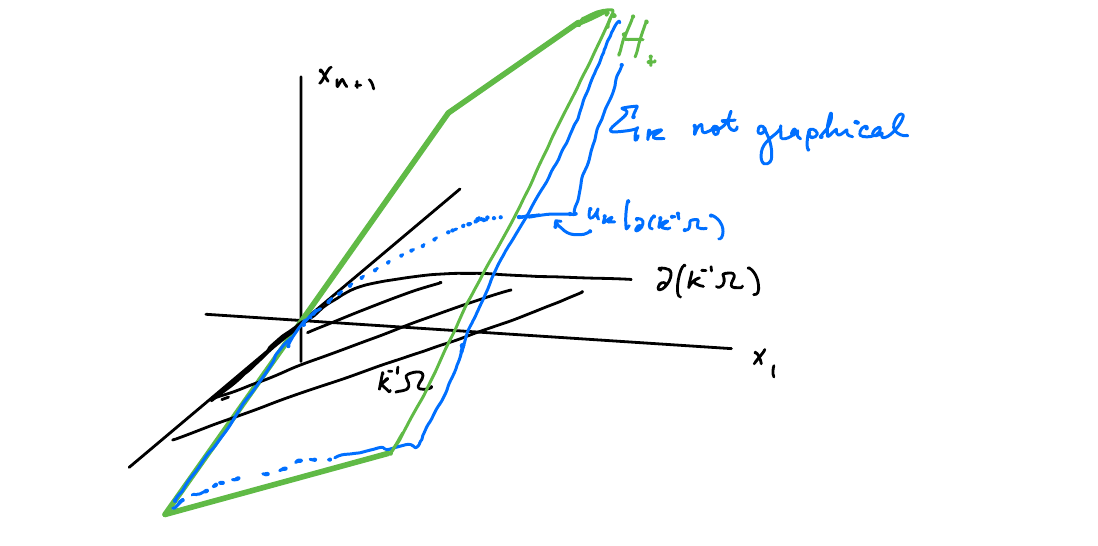}
\caption{$\Sigma_k$ is not graphical for $k$ large.}
\label{Fig4}
\end{center}
\end{figure}

We now turn to case (B). If $0 < A_+ < \infty$, then by Lemma \ref{Trapping}, $\Sigma_k$ contain points converging to a point in $\{x_1 > 0\}$, thus we contradict the graphicality of $\Sigma_k$ over $\{0 < x_1 < k^{-1}c\}$ for $k$ large. Assume now that $A_+ = \infty$. Let $B$ be a ball of radius one in $\{x_1 = 0\}$ that lies above $\Gamma$, and let $Q_h = \{-h < x_1 < h\} \times B$ for $h > 0$ small to be determined. Lemmas \ref{Trapping} and \ref{ABP} (appropriately rescaled) imply that, in $Q_h$, the hypersurfaces $\Sigma_k$ contain a sheet of anisotropic area approaching $|B_1|\Phi(-e_1)$ as $k \rightarrow \infty$ (see Figure \ref{Fig5}). Let $E_k = \{x_{n+1} < u_k(x),\, x \in k^{-1}\Omega\}$ and $F_k = E_k \backslash Q_h$. Then $\partial F_k$ are competitors for $\partial E_k$ in a neighborhood of $Q_h$ which for $k$ large have anisotropic area bounded above by that of $\partial E_k$ minus $|B_1|\Phi(-e_1)/2$ plus $C(n,\,K)h$ (the last term coming from the thin sides of the cylinder $Q_h$). For $h(n,\,K)$ small we contradict the minimizing property of $\partial E_k$. We conclude that $A_+ = 0$. The claim that $A_- = 0$ follows in the same way, after reflecting over $\{x_{n+1} = 0\}$ (and changing the functional accordingly).

\begin{figure}
 \begin{center}
    \includegraphics[scale=0.8, trim={20mm 170mm 0mm 15mm}, clip]{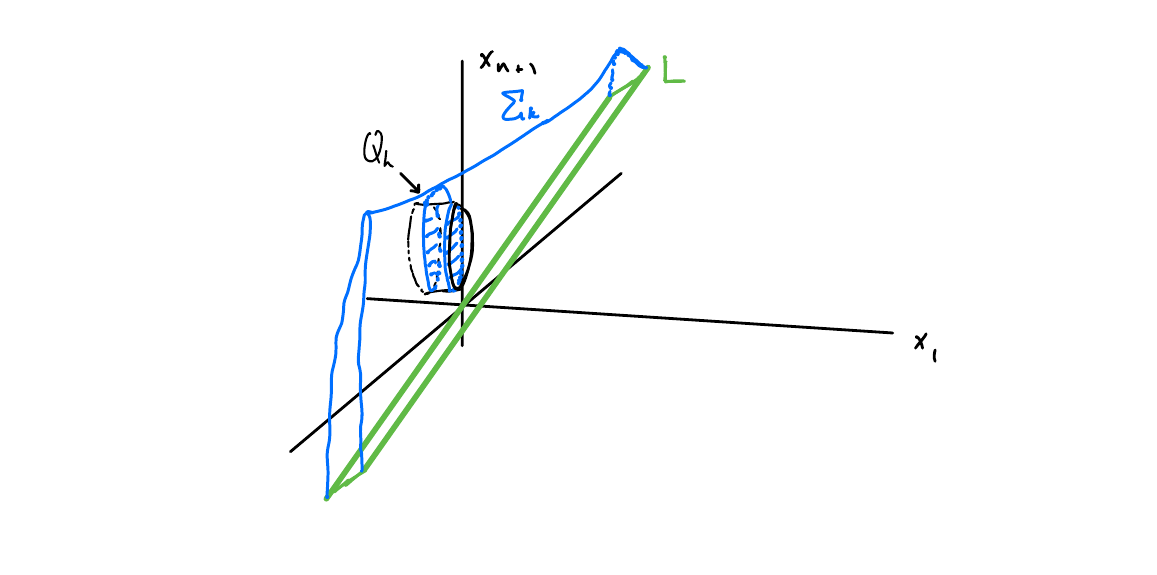}
\caption{$\Sigma_k$ has more anisotropic area in $Q_h$ than the thin side of $Q_h$.}
\label{Fig5}
\end{center}
\end{figure}

Finally we treat case (A). If $A_+$ and $A_-$ are in $\mathbb{R}$ and $A_- < A_+$, then Lemma \ref{Trapping} and Lemma \ref{ABP} imply that $\Sigma_k$ are simultaneously close to $H_{\pm}$ in measure for $k$ large, which contradicts the graphicality of $\Sigma_k$ in the $x_{n+1}$ direction. The problem is thus reduced (after possibly reflecting over $\{x_{n+1} = 0\}$) to ruling out the case that $A_+ = \infty$. We distinguish two sub-cases. The first is that $A_- \in \mathbb{R}$. Using Lemmas \ref{Trapping} and \ref{ABP} near both $H_+$ and $H_-$ we see that in $Q_h$, $\Sigma_k$ contain a sheet of anisotropic area approaching $|B_1|\Phi(-e_1)$, and another portion that projects in the $x_1$ direction to nearly all of $B$. For this one uses that for $k$ large, $\Sigma_k$ are very close in measure to $H_-$ on regions that get close to $\Gamma$ (see Figure \ref{Fig6}). Thus, the anisotropic area of $\Sigma_k$ in $Q_h$ is bounded from below by $|B_1|\Phi(-e_1) + c_0(n,\,K)$ as $k$ gets large. Taking $E_k$ and $F_k$ as in case (B) we again contradict minimality for $h(n,\,K)$ small, since removing $Q_h$ removes $\Sigma_k$ in $Q_h$ but adds at most the anisotropic area of the thin sides and one face of $Q_h$, which is $|B_1|\Phi(-e_1) + C(n,\,K)h$. The alternative is that $A_- = -\infty$. In this case Lemmas \ref{Trapping} and \ref{ABP} imply that $\Sigma_k$ have portions with anisotropic area nearly $|B_1|\Phi(-e_1)$ in $Q_h$ and $|B_1|\Phi(e_1)$ in $-Q_h$ for $k$ large. Using the graphicality of $\Sigma_k$ in the $x_{n+1}$ direction, we see by the pigeonhole principle that in at least one of $Q_h,\, -Q_h$, the hypersurface $\Sigma_k$ contains another portion that projects in the $x_1$ direction to nearly half of $B,\,-B$. We may assume that this happens in $Q_h$, after possibly reflecting over $\{x_{n+1} = 0\}$. Then the anisotropic area of $\Sigma_k$ in $Q_h$ is again bounded from below by $|B_1|\Phi(-e_1) + c_0(n,\,K)$ for $k$ large, and we contradict the minimizing property of $\Sigma_k$ as in the previous sub-case to complete the proof.

\begin{figure}
 \begin{center}
    \includegraphics[scale=0.8, trim={20mm 125mm 0mm 25mm}, clip]{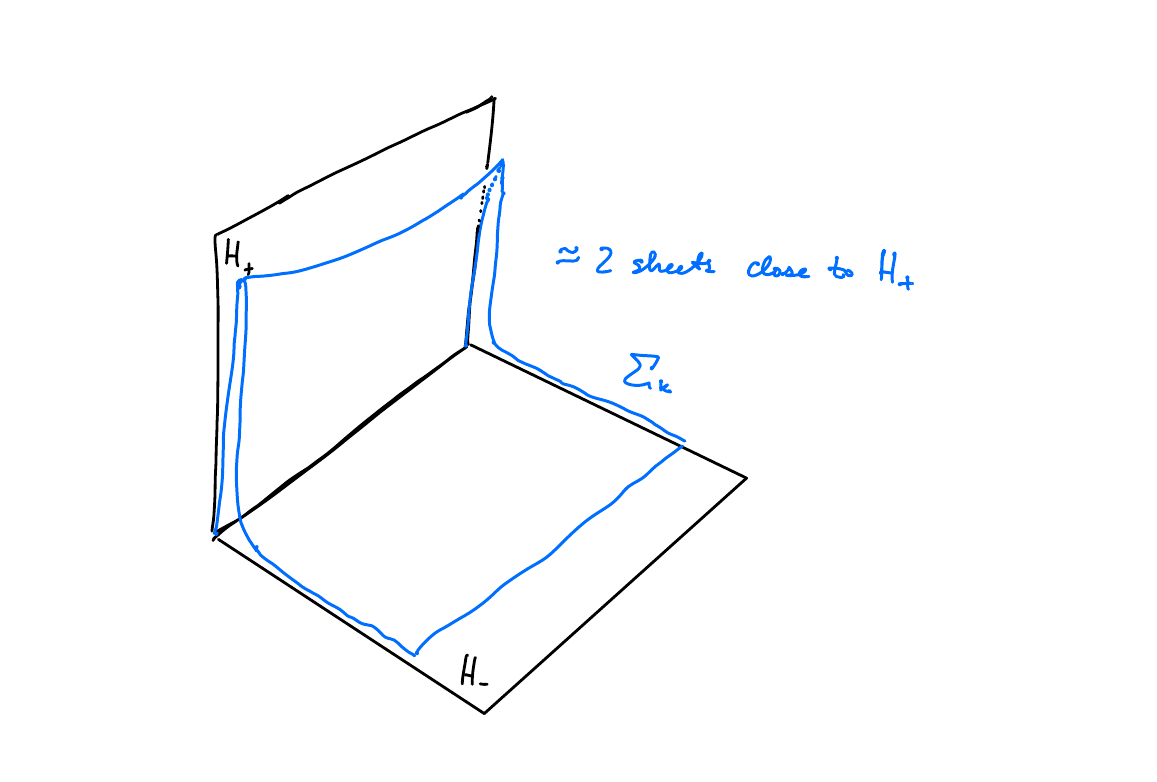}
\caption{$\Sigma_k$ contains nearly two vertical sheets for $k$ large.}
\label{Fig6}
\end{center}
\end{figure}
\end{proof}

\begin{rem}
The argument for case (C) in fact shows the linearity of $\Sigma$ when $\Omega$ is any domain in $\{x_1 > 0\}$ which, outside of a large ball, is contained in a convex cone that is not a half-space.
\end{rem}



\end{document}